\documentclass[12pt]{amsart}
\usepackage{geometry}   
\usepackage[colorlinks,citecolor = red, linkcolor=blue,hyperindex]{hyperref}
\usepackage{euscript,eufrak,verbatim, mathrsfs}
\usepackage[psamsfonts]{amssymb}
\usepackage{bbm}
\usepackage{graphicx}
\usepackage{float}
\usepackage{float, tikz}

 \usepackage{extarrows}
 \usepackage[all, cmtip]{xy}

\usepackage{upref, xcolor, dsfont}
\usepackage{amsfonts,amsmath,amstext,amsbsy, amsopn,amsthm}
\usepackage{enumerate}

\usepackage{url}

\usepackage{mathtools}
\usepackage{bookmark}

 \usepackage{euscript}
\usepackage{helvet}         
\usepackage{courier}        
\usepackage{type1cm}        
\usepackage{multicol}        
\usepackage[bottom]{footmisc}

\newtheorem{theorem}{Theorem}[section]
\newtheorem*{theorem*}{Theorem A}
\newtheorem{lemma}[theorem]{Lemma}

\newtheorem{proposition}[theorem]{Proposition}

\newtheorem*{definition*}{Definition}
\newtheorem*{remark*}{Remark}

\newtheorem*{observation*}{Observation}

\newtheorem*{assumption*}{Assumption}

\theoremstyle{definition}
\newtheorem{definition}{Definition}[section]

\theoremstyle{remark}
\newtheorem{remark}{Remark}[section]

\newcommand{\B}{\mathbb{B}}

\begin{document}

\title{The one-weight inequality for $\mathcal{H}$-harmonic Bergman projection}

\author
{Kunyu Guo}
\address{Kunyu Guo\\
	School of Mathematical Sciences, Fudan University\\
	Shanghai, 200433, China}
\email{kyguo@fudan.edu.cn}

\author
{Zipeng Wang}
\address{Zipeng Wang\\
	College of Mathematics and Statistics, Chongqing University\\
	Chongqing, 401331, China}
\email{zipengwang2012@gmail.com}

\author
{Kenan Zhang}
\address
{Kenan Zhang\\
	School of Mathematical Sciences, Fudan University\\
	Shanghai, 200433, China}
\email{knzhang21@m.fudan.edu.cn}

\thanks{}

\begin{abstract}
	Let $n\geqslant 3$ be an integer. For the Bekoll\'e-Bonami weight $\omega$ on the real unit ball $\B_n$, we obtain the following sharp one-weight estimate for the $\mathcal{H}$-harmonic Bergman projection: for $1<p<\infty$ and $-1<\alpha<\infty$,
	 \[||P_\alpha||_{ L^p(\omega d\nu_\alpha)\longrightarrow L^p(\omega d\nu_\alpha)}\leqslant C  [\omega]_{p,\alpha}^{\max\left\{1,\frac{1}{p-1}\right\}},
	 \]
	where $[\omega]_{p,\alpha}$ is the Bekoll\'e-Bonami constant. Our proof is inspired by the dyadic harmonic analysis, and the key ingredient involves the discretization of the Bergman kernel for the $\mathcal{H}$-harmonic Bergman spaces.
\end{abstract}


\maketitle

\setcounter{equation}{0}
\vspace{0.1in}
\section{Introduction}
Let $n\geqslant 3$ be an integer, and let $\B_n=\{x\in\mathbb{R}^n:|x|<1\}$ be the unit ball of the Euclidean space $\mathbb{R}^n$ where $|x|=(\sum_{i=1}^n|x_i|^2)^{1/2}$ is the standard Euclidean norm. We deonte by $\text{Aut}(\B_n)$ the M\"{o}bius transformation group on
the unit ball $\B_n$. It is also known that \cite[Chapter 2]{Stoll} the M\"{o}bius transformation group has the following explicit form
\[
\text{Aut}(\B_n)=\{A\circ\varphi_y:A\in O(n), y\in\B_n\},
\]
where $O(n)$  denotes the orthogonal group in $\mathbb{R}^n$, and
\[
\varphi_y(x)=\frac{y|x-y|^2+(1-|y|^2)(y-x)}{[x,y]^2},
\]
with $$[x,y]^2=|x-y|^2+(1-|y|^2)(1-|x|^2),\quad \forall x,y\in\B_n.$$

\vspace{0.1in}
Let $C^2(\B_n)$ be the space of complex-valued functions on the real unit ball $\B_n$ whose second-order partial derivatives are continuous on $\B_n$ and
consider the standard Laplace operator
\[
\Delta=\frac{\partial^2}{\partial x_1^2}+\cdots+\frac{\partial^2}{\partial x_n^2},
\]
then the invariant Laplace operator $\Delta_h$ is defined as a linear operator on $C^2(\B_n)$ such that for each $f\in C^2(\B_n)$ and $x\in\B_n$,
\[
	(\Delta_hf)(x)=\Delta(f\circ\varphi_x)(0).
\]
Moreover, by using the standard Laplace operator $\Delta$ and  the gradient operator 
$\nabla$, we have for each $x\in\B_n$:
\begin{equation}\label{E:ilaplace}
 (\Delta_hf)(x)=(1-|x|^2)^2\Delta f(x)+2(n-2)(1-|x|^2)\langle x,\nabla f(x)\rangle,
\end{equation}
where $\langle\cdot,\cdot\rangle$ 
denotes the standard Euclidean inner product on $\mathbb{R}^n$. Recall also the gradient operator 
$\nabla$ is given by
\[
\nabla=\left(\frac{\partial}{\partial x_1},\cdots,\frac{\partial}{\partial x_n}\right).
\]

\vspace{0.1in}
A complex-valued function $f\in C^2(\B_n)$ is said to be $\mathcal{H}$-harmonic if for each $x\in\B_n$,
\[
	(\Delta_hf)(x)=0.
\] 
Let $\mathcal{H}(\B_n)$ be the complex linear space of the ordinary  $\mathcal{H}$-harmonic functions.
If $n=2$, by the equality \eqref{E:ilaplace}, we have 
\[
(\Delta_hf)(x)=(1-|x|^2)^2\Delta f(x),
\]
for each $x\in\B_n$. 
In this case, the space $\mathcal{H}(\B_2)$ coincides with the space of harmonic functions on the unit disk.

Let $\nu$ be the normalized Lebesgue measure on $\B_n$ such that $\nu(\B_n)=1.$ 
For $-1<\alpha<\infty$, the standard weighted measure on the unit ball $\B_n$ is
\[
d\nu_\alpha(x)=c_\alpha (1-|x|^2)^\alpha d\nu(x),
\]
where $c_\alpha$ is the normalized constant such that $\nu_\alpha(\B_n)=1$.

\vspace{0.1in}
For $1\leqslant p<\infty$ and $-1<\alpha<\infty$, the $\mathcal{H}$-harmonic Bergman space $\mathcal{B}^p_\alpha$ is the  Banach space
\[
\mathcal{B}^p_\alpha =\{f\in\mathcal{H}(\B_n):\|f\|_{L^p(d\nu_\alpha)}<\infty\},
\]
where
\[
\|f\|_{L^p(d\nu_\alpha)}=\left[\int_{\B_n}|f(x)|^pd\nu_\alpha(x)\right]^{\frac{1}{p}}.
\]

\vspace{0.1in}

By the mean-value estimate, the point evaluations $f\mapsto f(x)$ at any $x\in\B_n$ are bounded on each $\mathcal{B}^2_\alpha,-1<\alpha<\infty$ \cite{BEY}. Therefore, the $\mathcal{H}$-harmonic Bergman space $\mathcal{B}^2_\alpha$ is a reproducing kernel Hilbert space. It follows that for any $x\in\B_n$, the reproducing kernel $\mathcal{R}_\alpha(x,\cdot)$ is the unique function in $\mathcal{B}^2_\alpha$, such that for all $f\in\mathcal{B}_\alpha^2$,
\begin{equation*}
	f(x)=\langle f(\cdot), \mathcal{R}_\alpha(x,\cdot)\rangle_{\alpha}=\int_{\B_n} \overline{\mathcal{R}_\alpha(x,y)}f(y) d\nu_\alpha(y).
\end{equation*}
One can show that $\mathcal{R}_\alpha(\cdot,\cdot)$ is a real-valued function. Hence, for any $f\in\mathcal{B}_\alpha^2$, we have
\begin{equation*}
	f(x)=\int_{\B_n} \mathcal{R}_\alpha(x,y)f(y) d\nu_\alpha(y)\qquad x\in\B_n.
\end{equation*}
For more details of  $\mathcal{H}$-harmonic functions and related function spaces on the unit ball, one can consult \cite{BEY,GJ,J06,Stoll,S12}.
\vspace{0.1in}

Recall that the $\mathcal{H}$-harmonic Bergman projection  $P_\alpha$ is the orthoganal projection from  $L^2(d\nu_\alpha)$ onto $ \mathcal{B}^2_\alpha$. Using the reproducing kernel property, we have
\[
	P_\alpha f(x)=\int_{\B_n} \mathcal{R}_\alpha(x,y)f(y)d\nu_\alpha(y),\quad \forall x\in\B_n.
\]
By \cite[Theorem 1.1]{Ureyen2023}, $P_\alpha$ is bounded on $L^p(d\nu_\alpha)$ for $1<p<\infty$ and $-1<\alpha<\infty$.

\vspace{0.1in}
In this paper, a weight $\omega$ is a non-negative integrable function on the real unit ball $\B_n$. The weighted Lebesgue space $L^p(\omega d\nu_\alpha)$ on $\B_n$ is given by
\[
L^p(\omega d\nu_\alpha)=\{f:\|f\|_{L^p(\omega d\nu_\alpha)}<\infty\},
\]
where
\[
\|f\|_{L^p(\omega d\nu_\alpha)}:=
\left[\int_{\B_n}|f(x)|^p\omega(x)d\nu_\alpha(x)\right]^{\frac{1}{p}}.
\]

\vspace{0.1in}
The weighted theory of Bergman projection-type operators on the unit ball, namely the operators whose integral kernels are reproducing kernels for some holomorphic or harmonic function spaces on the unit balls, has attracted much attention in recent years.  In particular, Bekoll\'e and Bonami \cite{B1981,B,BB1978}, as well as the ingenious work by Pott and Reguera \cite{PR2013}, present complete characterizations of the sharp linear norm estimate in the one-weight inequality for the holomorphic Bergman projections on the unit ball. In \cite{RTW}, Rahm, Tchoundja, and Wick extended these sharp one-weight inequalities to the setting of the Berezin transform on the unit ball.
Volberg and Wick \cite{VW} first applied non-homogeneous harmonic analysis techniques to weighted inequalities for Bergman projections, obtaining deep applications concerning Carleson measures for function spaces on the unit ball.
For further developments on weighted inequalities for holomorphic Bergman projections and their applications, we refer the reader to \cite{APR,GW2024,HWW,HW24,S2023}.

\vspace{0.1in}

Motivated by these developments on the weighted inequalities of holomorphic Bergman projection, it is natural to study the following one-weight inequality for $\mathcal{H}$-harmonic Bergman projection: 

\vspace{0.2cm}

\noindent \textbf{Problem.} Let $1<p<\infty$ and $-1<\alpha<\infty$. Characterize the weights $\omega$ for which the $\mathcal{H}$-harmonic Bergman projection $P_\alpha$ is bounded on $L^p(\omega d\nu_\alpha)$?
\vspace{0.1in}

Observe that the $\mathcal{H}$-harmonic Bergman projection on $\B_2$ coincides with the classical harmonic Bergman projection on the unit disk.
Although the results of Bekoll\'e-Bonami and Pott-Reguera were initially established for the holomorphic Bergman projection, their results extend naturally to the harmonic setting. Consequently, by the work of Bekoll\'e-Bonami and Pott-Reguera, the one-weight inequality for the $\mathcal{H}$-harmonic Bergman projection on $\B_2$ is characterized by the Bekoll\'e-Bonami class of weights. In the case $n\geqslant 3$, it is therefore natural to conjecture that the Bekoll\'e-Bonami weights continue to ensure the boundedness of the $\mathcal{H}$-harmonic Bergman projection on $L^p(\omega d\nu_\alpha)$. In this note, we shall confirm this folk conjecture.

\vspace{0.1in}
For a set $E$ in the unit sphere $\partial\B_n=\{x\in\mathbb{R}^n:|x|=1\}$,
a Carleson box $\widehat{E}$ in the unit ball associated with the set $E$ is defined  as
\[
	\widehat{E}:=\left\{z\in\B_n:\frac{z}{|z|}\in E,1-r<|z|<1\right\},
\]
where $r=\textup{diam}(E)/2$ and $\textup{diam}(E)=\max\{|x-y|:x,y\in E\}$.
\vspace{0.1in}

Let $1<p<\infty$ and $-1<\alpha<\infty$, a weight $\omega$ on $\B_n$ is said to be a Bekoll\'e-Bonami $B_{p,\alpha}$ weight if 
the Bekoll\'e-Bonami constant $[\omega]_{p,\alpha}$ is finite. Here $[\omega]_{p,\alpha}$ is defined as
 \begin{equation*}
[\omega]_{p,\alpha}:=\sup\limits_{B\subset\partial \B_n}\frac{|\widehat{B}|_{\omega,\alpha}}{|\widehat{B}|_\alpha}\left(\frac{|\widehat{B}|_{\omega^{1-p'},\alpha}}{|\widehat{B}|_\alpha}\right)^{p-1},
\end{equation*}
where 
$\frac{1}{p}+\frac{1}{p'}=1$, $|\widehat{B}|_{\alpha}=\nu_\alpha(\widehat{B})$, $|\widehat{B}|_{\omega,\alpha}=\int_{\widehat{B}}\omega d\nu_\alpha$ and that the supremum runs over all Euclidean balls $B\subset\partial \B_n$. 

\vspace{0.1in}

Our main result is stated as follows
\begin{theorem}\label{Thm-key}
	Let $1<p<\infty$ and $-1<\alpha<\infty$. If $\omega$ is a $ B_{p,\alpha}$ weight, then the $\mathcal{H}$-harmonic Bergman projection 
	$
	P_\alpha $ is bounded on  $L^p(\omega d\nu_\alpha)$. More precisely, there exists a constant $C=C(n,p,\alpha)$ such that
	 \begin{equation*}
	 ||P_\alpha||_{ L^p(\omega d\nu_\alpha)\to L^p(\omega d\nu_\alpha)}\leqslant C  [\omega]_{p,\alpha}^{\max\left\{1,\frac{1}{p-1}\right\}}.
	\end{equation*}
\end{theorem}
\vspace{0.1in}

\begin{remark}
Our strategy of proving Theorem \ref{Thm-key} is inspired by Pott-Reguera's recent beautiful work on the sharp one-weight estimate for Bergman projection \cite{PR2013}. However, 
since there is no explicit closed form of the kernel of the $\mathcal{H}$-harmonic Bergman projection. To overcome this difficulty, we construct suitable dyadic positive operators on appropriately chosen Carleson box systems, which allow us to effectively dominate the original operator.
\end{remark}

The paper is organized as follows. In Section 2, based on certain dyadic Carleson boxes on the unit ball, we provide a discretization of the $\mathcal{H}$-harmonic Bergman projection. The proof of the
main result is presented in Section 3, which is divided into two parts. For the first part, in Subsection 3.1, we prove that case for $p=2$. While in Subsection 3.2 we deal with the case for $p\not=2$ by Rubio de Francia's extrapolation method. In the final section, by constructing specific examples, we demonstrate that the norm estimate in the one-weight inequality for the $\mathcal{H}$-harmonic Bergman projection is sharp.

\vspace{0.2in}

\section{Carleson systems and the discretization}
\vspace{0.1cm}
\subsection{The Carleson box on the real unit ball}
For two points $x$ and $y$ in $\partial\B_n$, let $\rho(x,y)$ denote the angle between them, defined by
\[
\rho(x,y)=\arccos(\langle x,y \rangle ).
\]

\begin{proposition}\label{metric-angle}
	The function $\rho(x,y)$ defines a metric on
	$\partial\B_n$ that is equivalent to the standard Euclidean distance on $\partial\B_n$.
\end{proposition}

\begin{proof}
To verify that $\rho(\cdot,\cdot)$ is a metric on $\partial\B_n$,  it suffices to establish the triangle inequality:  for any $x,y,z\in\partial\B_n$, 
	$$
	\rho(x,y)\leqslant\rho(x,z)+\rho(z,y).
	$$
	By the definition of the function $\rho(\cdot,\cdot)$,  the triangle inequality holds if and only if 
	\[
	\cos\rho(x,y)\geqslant\cos(\rho(x,z)+\rho(z,y)).
	\]
	Note that $|x|=|y|=|z|=1$, by the Binet-Cauchy identity,
	\[
	\langle x, y \rangle=\langle x, z \rangle\langle z, y \rangle+\sum\limits_{1\leqslant i< j\leqslant n}(x_iz_j-x_jz_i)(y_iz_j-y_jz_i),
	\]
	we have
	\begin{equation*}
		\begin{aligned}
			\cos\rho(x,y)&=\langle x, z \rangle\langle z, y \rangle+\sum\limits_{1\leqslant i< j\leqslant n}(x_iz_j-x_jz_i)(y_iz_j-y_jz_i)\\
			&=\cos\rho(x,z)\cos\rho(z,y)+\sum\limits_{1\leqslant i< j\leqslant n}(x_iz_j-x_jz_i)(y_iz_j-y_jz_i)\\
			&\geqslant\cos\rho(x,z)\cos\rho(z,y)-\left(\sum\limits_{1\leqslant i< j\leqslant n}|x_iz_j-x_jz_i|^2\right)^{\frac{1}{2}}\left(\sum\limits_{1\leqslant i< j\leqslant n}|y_iz_j-y_jz_i|^2\right)^{\frac{1}{2}}.\\
		\end{aligned}
	\end{equation*}
	By using the Binet-Cauchy identity again, we get
	\[
	\sum\limits_{1\leqslant i< j\leqslant n}|x_iz_j-x_jz_i|^2=|x|^2|z|^2-|\langle x, z \rangle|^2=\sin^2\rho(x,z).
	\]
It follows that
	\begin{equation*}
		\begin{aligned}
			\cos\rho(x,y)\geqslant\cos\rho(x,z)\cos\rho(z,y)-\sin\rho(x,z)\sin\rho(z,y)=\cos(\rho(x,z)+\rho(z,y)).
		\end{aligned}
	\end{equation*}
	
Next, we show that the metric $\rho(\cdot,\cdot)$ on $\partial\B_n$  is equivalent to the Euclidean distance.
For any two points $x,y\in \partial\B_n$, we observe that
	\begin{equation*}
		\begin{aligned}
			|x-y|^2
			&=|x|^2+|y|^2-2\langle x,y \rangle\\
			&=2(1-\cos\rho(x,y))\\
			&=4\left(\sin\frac{\rho(x,y)}{2}\right)^2.
		\end{aligned}
	\end{equation*}
	Note that for any $0\leqslant\theta\leqslant\pi/2$, we have
	\[
		\frac{2}{\pi}\theta\leqslant\sin\theta\leqslant\theta;
	\]
	and for any $x,y\in\partial\B_n$, we have
	\[
		0\leqslant\rho(x,y)\leqslant\pi.
	\]
	Hence,
	\begin{equation}\label{metric-equi}
		\frac{2}{\pi}\rho(x,y)\leqslant |x-y|\leqslant\rho(x,y).
	\end{equation}
	This completes the proof.	
\end{proof}

\vspace{0.1in}
Let $x\in\partial\B_n$ and $0<r<1$. We consider a ball $B_\rho(x,r)$ on the unt sphere $\partial\B_n$ that is defined by
\[
	B_\rho(x,r)=\{y\in\partial\B_n:\rho(x,y)<r\}
\]
and its associated Carleson box $\widehat{B}_\rho(x,r)$ on the unit ball
\begin{equation*}
	\widehat{B}_\rho(x,r):=\left\{z\in\B_n:\frac{z}{|z|}\in B_\rho(x,r),1-r<|z|<1\right\}.
\end{equation*}

By Proposition \ref{metric-angle}, the Bekoll\'e-Bonami constant $[w]_{p,\alpha}$ is equivalent, up to a constant,  to the following quantity:
	\begin{equation}\label{new-bp}
	\sup\limits_{B_\rho\subset\partial \B_n}\frac{|\widehat{B}_\rho|_{\omega,\alpha}}{|\widehat{B}_\rho|_\alpha}\left(\frac{|\widehat{B}_\rho|_{\omega^{1-p'},\alpha}}{|\widehat{B}_\rho|_\alpha}\right)^{p-1}<\infty,
	\end{equation}
	where $\frac{1}{p}+\frac{1}{p'}=1$ and the supremum runs over all balls $B_\rho$ satisfying $B_\rho(x,r)\subset \partial \B_n$.

\begin{lemma}\label{area-ball}
Let $n\geqslant 3$ and $0<r<1$. For a ball $B_\rho(x,r)\subset\partial\B_n$, there exists positive constants $C_1=C_1(n,\alpha)$ and $C_2=C_2(n)$ such that
	\begin{equation}\label{bi-ball-area}
	 C_1\frac{c_\alpha}{\alpha+1}r^{n+\alpha}(2-r)^{\alpha+1}\leqslant |\widehat{B}_\rho(x,r)|_\alpha \leqslant C_2\frac{c_\alpha}{\alpha+1}r^{n+\alpha}(2-r)^{\alpha+1}.
	\end{equation}
\end{lemma}

\begin{proof}
Given a ball $B_\rho (x,r)\subset\partial\B_n$, by the inequality \eqref{metric-equi}, 
\[
	B(x,2r/\pi)\subset B_\rho(x,r)\subset B(x,r).
\]
Here, $B(x,r) = \{y \in \partial\B_n : |x - y| < r\}$ denotes the standard Euclidean ball centered at $x$ with radius $r$.
Therefore, there exists positive constants $c_1=c_1(n)$ and $c_2=c_2(n)$ such that 
\[
	c_1 r^{n-1}\leqslant \int_{B_\rho(x,r)} \sigma(\xi)\leqslant c_2 r^{n-1},
\]
where $d\sigma$ is normalized sphere area measure.
Hence,
\begin{equation*}
	\begin{aligned}
		|\widehat{B}_\rho(x,r)|_\alpha&= c_\alpha\int_{\widehat{B}_\rho(x,r)}(1-|z|^2)^\alpha d\nu(z)\\
		&=n c_\alpha \int_{B_\rho(x,r)} d\sigma(\xi)\int_{1-r}^1t^{n-1}(1-t^2)^\alpha dt\\
		&\leqslant nc_2 c_\alpha r^{n-1}\int_{1-r}^1t^{n-1}(1-t^2)^\alpha dt.
	\end{aligned}
\end{equation*}
Since $1-r<t<1$ and $n\geqslant 3$, let $C_2=n c_2$, we have 
\begin{equation*}
	\begin{aligned}
	|\widehat{B}_\rho(x,r)|_\alpha&\leqslant C_2 c_\alpha r^{n-1}\int_{1-r}^1t(1-t^2)^\alpha dt\\
		&=C_2 \frac{c_\alpha}{\alpha+1}r^{n+\alpha}(2-r)^{\alpha+1}.
	\end{aligned}
\end{equation*}
This establishes the upper bound in \eqref{bi-ball-area}. 
\vspace{0.1in}

Next, we estimate the lower bound. For $0<r\leqslant1/2,$ we have 
\[
	\frac{1}{2}\leqslant 1-r<t<1.
\]
Therefore,
\begin{equation*}
	\begin{aligned}
	|\widehat{B}_\rho(x,r)|_\alpha
	&\geqslant n c_1 c_\alpha r^{n-1}\int_{1-r}^1t^{n-1}(1-t^2)^\alpha dt\\
	&\geqslant \frac{nc_1}{2^{n-2}}c_\alpha r^{n-1}\int_{1-r}^1t(1-t^2)^\alpha dt\\
	&\geqslant \frac{nc_1}{2^{n-2}}\frac{c_\alpha}{\alpha+1}r^{n+\alpha}(2-r)^{\alpha+1}.
	\end{aligned}
\end{equation*}
For $1/2\leqslant r<1,$ we obtain
\begin{equation*}
	\begin{aligned}
		|\widehat{B}_\rho(x,r)|_\alpha&\geqslant n c_1 c_\alpha r^{n-1}\int_{\frac{1}{2}}^1t^{n-1}(1-t^2)^\alpha dt\\
		&\geqslant\frac{nc_1}{2^{n-2}}c_\alpha r^{n-1}\int_{\frac{1}{2}}^1t(1-t^2)^\alpha dt\\
		&=\frac{nc_1}{2^{n-2}}\left(\frac{3}{4}\right)^{\alpha+1}\frac{c_\alpha}{\alpha+1}r^{n-1}.
	\end{aligned}
\end{equation*}
Since $0<2r-r^2<1$, we have
\begin{equation*}
	\begin{aligned}
		|\widehat{B}_\rho(x,r)|_\alpha
		&\geqslant \frac{nc_1}{2^{n-2}}\left(\frac{3}{4}\right)^{\alpha+1}\frac{c_\alpha}{\alpha+1}r^{n-1}(2r-r^2)^{\alpha+1}\\
		&= \frac{nc_1}{2^{n-2}}\left(\frac{3}{4}\right)^{\alpha+1}\frac{c_\alpha}{\alpha+1}r^{n+\alpha}(2-r)^{\alpha+1}.
	\end{aligned}
\end{equation*}
Let 
\begin{equation*}
	\begin{aligned}
		C_1&=\min\left\{\frac{nc_1}{2^{n-2}}, \frac{nc_1}{2^{n-2}}\left(\frac{3}{4}\right)^{\alpha+1}\right\}\\
		&=\frac{nc_1}{2^{n-2}}\left(\frac{3}{4}\right)^{\alpha+1}.
	\end{aligned}
\end{equation*}
Then, for any $x\in\partial\B_n$ and $0<r<1$,
\[
	|\widehat{B}_\rho(x,r)|_\alpha\geqslant C_1 \frac{c_\alpha}{\alpha+1}r^{n+\alpha}(2-r)^{\alpha+1}.
\]
This is the lower bound in \eqref{bi-ball-area} and we complete the whole proof.
\end{proof}
\subsection{An example of $B_{2,\alpha}$}\label{S:sharp-example}
By using Lemma \ref{area-ball}, in this subsection, we obtain a precise estimate for the Bekoll\'e-Bonami constant $[\omega]_{2,\alpha}$  for 
\[
\omega(x)=(1-|x|^2)^s, \quad x\in \B_n,
\]
where  $-1<\alpha<\infty$, $0<\delta<1$ and $s=(\alpha+1)(1-\delta)$. Such examples play an important role in the sharp estimate for the norm of $P_\alpha$ in Section \ref{S:sharp}.

Indeed, for any $x\in\partial\B_n$ and $0<r<1$, consider the ball $B_\rho(x,r)$, we have
\[
	|\widehat{B}_\rho(x,r)|_{\omega,\alpha}=\frac{c_\alpha}{c_{\alpha+s}}|\widehat{B}_\rho(x,r)|_{\alpha+s} \quad \text{ and } \quad
	|\widehat{B}_\rho(x,r)|_{\omega^{-1},\alpha}=\frac{c_\alpha}{c_{\alpha-s}}|\widehat{B}_\rho(x,r)|_{\alpha-s}.
\]
Observe that 
\[
	\alpha+s=(\alpha+1)(2-\delta)-1>-1 \quad \text{ and } \quad \alpha-s=(\alpha+1)\delta-1>-1,
\]
and the upper estimate in \eqref{bi-ball-area} holds for any $-1<\alpha<\infty$, we have
	\[
		|\widehat{B}_\rho(x,r)|_{\omega,\alpha}\leqslant C_2\frac{c_\alpha}{s+\alpha+1}r^{n+s+\alpha}(2-r)^{s+\alpha+1},
	\]
	and
	\[
		|\widehat{B}_\rho(x,r)|_{\omega^{-1},\alpha}\leqslant C_2\frac{c_\alpha}{-s+\alpha+1}r^{n-s+\alpha}(2-r)^{-s+\alpha+1}.
	\]
Besides, by the lower bound in \eqref{area-ball},
\[
|\widehat{B}_\rho(x,r)|_{\alpha}\geqslant C_1 \frac{c_\alpha}{\alpha+1}r^{n+\alpha}(2-r)^{\alpha+1}. 
\]
Note also that we can choose $C_1,C_2$ are independent of $\delta$. Hence
\begin{equation}\label{example-bp}
	\begin{aligned}
		\sup\limits_{B_\rho(x,r)\subset \partial \B_n}\frac{|\widehat{B}_\rho(x,r)|_{\omega,\alpha}|\widehat{B}_\rho(x,r)|_{\omega^{-1},\alpha}}{|\widehat{B}_\rho(x,r)|_\alpha^2}
		&\leqslant \left(\frac{C_2}{C_1}\right)^2\frac{1}{(s+\alpha+1)(-s+\alpha+1)}\\
		&\leqslant \left(\frac{C_2}{C_1(\alpha+1)}\right)^2\frac{1}{\delta}.
	\end{aligned}
\end{equation}
It follows from \eqref{new-bp} and \eqref{example-bp} that there is a constant $c_3=c_3(n,\alpha)$ such that
\begin{equation}\label{example-weight}
	[\omega]_{2,\alpha}\leqslant\frac{c_3}{\delta}.
\end{equation}

\vspace{0.1cm}
\subsection{The Carleson box system on the real unit ball}
To discretize the $\mathcal{H}$-harmonic Bergman projection $P_\alpha$, we need dyadic systems on $(\partial\B_n,\rho)$.
A collection of sets 
\[
	\mathscr{D}=\{Q_{k,i}\subset\partial\B_n: 1\leqslant k<\infty,1\leqslant i\leqslant M(k)\}
\]
is called a dyadic system on $(\partial\B_n,\rho)$ if 
\begin{enumerate}[{(i)}]
	\item 
	for each integer $k\geqslant 1$, $\{Q_{k,i}\}_{i=1}^{M(k)}$ is a disjoint covering of $\partial\B_n$, i.e.,
	\begin{itemize}
		\item $\partial\B_n=\bigcup\limits_{i=1}^{M(k)} Q_{k,i}$ and 
		\item for all $1\leqslant i\neq j\leqslant M(k), Q_{k,i}\cap Q_{k,j}=\emptyset $.
	\end{itemize}
	\item if $1\leqslant k\leqslant l$, then either $Q_{l,j}\subset Q_{k,i}$ or $Q_{k,i}\cap Q_{l,j}=\emptyset$.
\end{enumerate}
\vspace{0.1in}
Furthermore, we say the dyadic system $\mathscr{D}$ is associated to a triple of positive constants $(\eta,\kappa_0,\kappa_1)$ if there exists a point set 
\[
	\mathcal{P}=\{x_{k,i}\in\partial\B_n:1\leqslant k<\infty,1\leqslant i\leqslant M(k)\}
\] 
such that 
\begin{equation}\label{dyadic-to-ball}
	B_\rho(x_{k,i},\kappa_0\eta^k)\subset Q_{k,i}\subset B_\rho(x_{k,i},\kappa_1\eta^k),
\end{equation}
for each $k\geqslant 1, 1\leqslant i\leqslant M(k)$.

\vspace{0.1in}
By Hytonen and Kairema's reuslt (\cite[Theorem 4.1]{Hytonen2012}), we have the following result on metric space $(\partial\B_n,\rho)$:

\begin{proposition}\label{dyadic-thm}
There is a finite collection of dyadic systems $\{\mathscr{D}_{t}\}_{t=1}^{N}$ associated to the triple $(1/96,1/12,4)$ on $(\partial\B_n,\rho)$ such that for any ball $B_\rho \subset \partial\B$, there exists $1\leqslant t\leqslant N$ and a cube $Q\in\mathscr{D}_{t}$ with
\begin{equation*}
	B_\rho \subset Q \quad \mbox{and}\quad \textup{diam}_\rho (Q)\leqslant C_3\textup{diam}_\rho(B_\rho),
\end{equation*}
where $\textup{diam}_\rho(Q)=\max\{\rho(x,y):x,y\in Q\}$ and the constants $N$ and $C_3$ only depend on the dimension $n$.
\end{proposition}
\begin{remark}
For the case $n=2$, the corresponding dyadic systems in Proposition \ref{dyadic-thm} are two shifted systems of the standard dyadic system on the unit circle. One can consult Mei's elegent arguments in \cite{Mei}.
\end{remark}
For $1\leqslant t\leqslant N$ and a cube $Q\in\mathscr{D}_t$, the Carleson box associated to $Q$ is defined as:
\[
\widehat{Q}:=\left\{z\in\B_n:\frac{z}{|z|}\in Q,1-\frac{1}{2}\textup{diam}_\rho(Q)<|z|<1\right\},
\]
then the dyadic system $\mathscr{D}_t$ induces a Carleson box system $\mathscr{Q}_t$ in the unit ball by
\[
\mathscr{Q}_t=\{\widehat{Q}:Q\in\mathcal{D}_t\}.
\]
In addition, let $0<\varepsilon<1$, we define the $\varepsilon$-Carleson box associated to $Q$ as
\[
\widehat{Q}_\varepsilon=\left\{z\in\B_n:\frac{z}{|z|}\in Q,1-\frac{\varepsilon}{2}\textup{diam}_\rho(Q)<|z|<1\right\}.
\]

\begin{lemma}\label{disjoint}
Let $1\leqslant t\leqslant N$ and $k\geqslant 1$. If there exist cubes $Q_{k+1,j},Q_{k,i}\in\mathscr{D}_t$ with $Q_{k+1,j}\subset Q_{k,i}$, then
	\[
		\widehat{Q}_{k+1,j}\subset\left(\widehat{Q}_{k,i}\right)_{1/2}.
	\]
\end{lemma}
\begin{proof}
	Fix an integer $1\leqslant t\leqslant N$, recall $\mathscr{D}_t$ is a dyadic system in $(\partial\B_n,\rho)$ with the triple $(1/96,1/12,4)$.
	For each positive integer $k\geqslant 1$, by Proposition \ref{dyadic-thm}, we have 
	\[
		\frac{1}{12}\cdot\left(\frac{1}{96}\right)^k\leqslant \frac{1}{2}\textup{diam}_\rho(Q_{k,i})\leqslant 4\cdot\left(\frac{1}{96}\right)^k.
	\] 
	Consequently, we have the inclusion
	\[
	\widehat{Q}_{k+1,j}\subset\left\{z\in\B_n:\frac{z}{|z|}\in Q_{k,i},1-4\cdot\left(\frac{1}{96}\right)^{k+1}<|z|<1\right\}.
	\]
	Noting that
	\[
	\left\{z\in\B_n:\frac{z}{|z|}\in Q_{k,i},1-\frac{1}{24}\cdot\left(\frac{1}{96}\right)^k<|z|<1\right\}
	\subset\left(\widehat{Q}_{k,i}\right)_{1/2},
	\]
	we conclude that \[\widehat{Q}_{k+1,j}\subset\left(\widehat{Q}_{k,i}\right)_{1/2},
	\]
	which completes the proof.
\end{proof}

\begin{lemma}\label{doubling}
	Let $-1<\alpha<\infty$.
	For an integer $1\leqslant t\leqslant N$ and a cube $Q\in\mathscr{D}_t$, there exists a constant $C_4=C_4(n,\alpha)$, such that 
	\[
		|\widehat{Q}|_\alpha\leqslant C_4|\widehat{Q}\backslash\widehat{Q}_{1/2}|_\alpha.
	\]
\end{lemma}
\begin{proof}
Pick a cube $Q\in\mathscr{D}_t$, we shall find a positive constant $C_4$ that only depneds on $n$ and $\alpha$ such that
	\[
		\frac{|\widehat{Q}\backslash\widehat{Q}_{1/2}|_\alpha}{|\widehat{Q}|_\alpha}\geqslant\frac{1}{C_4}>0.
	\]
	Let $r=\textup{diam}_\rho(Q)/2$ and consider the function $h(t)=t^{n-1}(1-t^2)^\alpha$ on $[0,1]$. We have
	\begin{equation*}
		\begin{aligned}
			\frac{|\widehat{Q}\backslash\widehat{Q}_{1/2}|_\alpha}{|\widehat{Q}|_\alpha}
			=&\frac{n c_\alpha \int_Q d\sigma(\xi)\int_{1-r}^{1-\frac{r}{2}}t^{n-1}(1-t^2)^\alpha dt}{n c_\alpha \int_Q d\sigma(\xi)\int_{1-r}^1t^{n-1}(1-t^2)^\alpha dt}\\
			=&\frac{\int_{1-r}^{1-\frac{r}{2}}h(t) dt}{\int_{1-r}^1h(t) dt}.
		\end{aligned}
	\end{equation*}
	For $0<r\leqslant1,$ set 
	\[g(r)=\frac{\int_{1-r}^{1-\frac{r}{2}}h(t) dt}{\int_{1-r}^1h(t) dt}.\]
	Observe that
	\begin{equation*}
		\begin{aligned}
			\lim\limits_{r\rightarrow0} g(r)
			=&\lim\limits_{r\rightarrow0}\frac{-\frac{1}{2}h\left(1-\frac{r}{2}\right)+h(1-r)}{h(1-r)}\\
			=&1-\lim\limits_{r\rightarrow0}\frac{\left(1-\frac{r}{2}\right)^{n-1}\left(r-\frac{r^2}{4}\right)^\alpha}{2(1-r)^{n-1}(2r-r^2)^\alpha}\\
			=&1-\frac{1}{2}\cdot\left(\frac{3}{4}\right)^\alpha>0,
		\end{aligned}
	\end{equation*}
	and 
	\[
		g(1)=\frac{\int_{0}^{\frac{1}{2}}h(t) dt}{\int_{0}^{1}h(t) dt}>0.
	\]
	Let $g(0)=\lim\limits_{r\rightarrow0} g(r)$.
	Then, $g(r)$ is continuous on $[0,1]$.
	Hence, there exists $0<r_0<1$ such that 
	\[
		g(r_0)=\min\limits_{0\leqslant r\leqslant1}g(r).
	\]
	We claim that $g(r_0)>0$. Otherwise, if \[
	g(r_0)=0,
	\]
	then
	\[g(r_0)=\frac{\int_{1-r_0}^{1-\frac{r_0}{2}}h(t) dt}{\int_{1-r_0}^1h(t) dt}=0.\]
	Since the function $h$ is positive and integrable on $(0,1)$, we have
	\[
		\int_{1-r_0}^{1-\frac{r_0}{2}}h(t) dt=0.
	\]
	On the other hand, note that
	\[
	 h(t)>0,\quad\forall t\in[1-r_0,1-r_0/2]
	\]
	this yields contradiction and we get
	\[g(r_0)=\min\limits_{0\leqslant r\leqslant1}g(r)>0.\]
	Let 
	\[
		C_4=\frac{1}{\min\limits_{0\leqslant r\leqslant1}g(r)},
	\]
we get the desired lower bound, and	complete the whole proof.
\end{proof}

Let $1<p<\infty$ and $-1<\alpha<\infty$. 
Let $\mathscr{Q}_t,t=1,\cdots N$ be the Carleson dyadic system in Proposition \ref{dyadic-thm}.
A weight $\omega$ is said to be a Bekoll\'e-Bonami weight with respect to the Carleson dyadic system $\mathscr{Q}_t$ if
\[
	[\omega]_{p,\alpha,\mathscr{D}_t}:=\sup\limits_{Q\in\mathscr{D}_t}\frac{|\widehat{Q}|_{\omega,\alpha}}{|\widehat{Q}|_\alpha}\left(\frac{|\widehat{Q}|_{\omega^{1-p'},\alpha}}{|\widehat{Q}|_\alpha}\right)^{p-1}<\infty,
\]
\begin{definition}
 A weight $\omega$ on $B_n$ is called a Bekoll\'e-Bonami $B_{p,\alpha,\mathscr{D}}$ weight associated to Carleson box systems $\mathscr{Q}=\cup_{t=1}^N\mathscr{Q}_t$ if
	\[
	[\omega]_{p,\alpha,\mathscr{D}}:=\max_{t:1\leqslant t\leqslant N} [\omega]_{p,\alpha,\mathscr{D}_t}<\infty.
	\]
\end{definition}

\begin{lemma}\label{weight-equ}
	Let $\omega$ be a weight on the unit ball $\partial\B_n$. Then $\omega$ is a $B_{p,\alpha}$ weight if and only if  $\omega$ is a $B_{p,\alpha,\mathscr{D}}$ weight.
\end{lemma}
\begin{proof}
	Suppose that $\omega$ is a $B_{p,\alpha}$ weight. Then
	$
	[\omega]_{p,\alpha}<\infty.
	$ 
	For each $1\leqslant t\leqslant N$, let
	\[
	\mathcal{P}_t=\{x_{k,i}\in\partial\B_n:1\leqslant k<\infty,1\leqslant i\leqslant M(k)\}
	\] 
	be the point set associated to the dyadic systems $\mathscr{D}_t$ in Proposition \ref{dyadic-thm}.
	By \eqref{dyadic-to-ball}, for any $Q\in\mathscr{D}_t$, there exist an integer $k\geqslant 1$ and a point $x\in\mathcal{P}_t$ such that
	\[
		B_\rho\left(x,\frac{1}{12\cdot96^k}\right)\subset Q\subset B_\rho\left(x,\frac{4}{96^k}\right).
	\]
Denote $B_1:=B_\rho\left(x,1/(12\cdot96^k)\right)$ and $B_2:=B_\rho\left(x,4/(96^k)\right)$.
	Then
	\begin{equation*}
		\begin{aligned}
			\frac{|\widehat{Q}|_{\omega,\alpha}}{|\widehat{Q}|_\alpha}\left(\frac{|\widehat{Q}|_{\omega^{1-p'},\alpha}}{|\widehat{Q}|_\alpha}\right)^{p-1}
			&\leqslant\frac{|\widehat{B_2}|_{\omega,\alpha}|\widehat{B_2}|_{\omega^{1-p'},\alpha}^{p-1}}{|\widehat{B_1}|_\alpha^p}.
		\end{aligned}
	\end{equation*}
	Hence, by \eqref{new-bp} and Lemma \ref{area-ball}, there is a constant $c_4=c_4(n,p,\alpha)>0$ such that
	\[
		\frac{|\widehat{Q}|_{\omega,\alpha}}{|\widehat{Q}|_\alpha}\left(\frac{|\widehat{Q}|_{\omega^{1-p'},\alpha}}{|\widehat{Q}|_\alpha}\right)^{p-1}\leqslant c_4\frac{|\widehat{B_2}|_{\omega,\alpha}|\widehat{B_2}|_{\omega^{1-p'},\alpha}^{p-1}}{|\widehat{B_2}|_\alpha^p}\leqslant c_4[w]_{p,\alpha}.
	\]
	Consequently, 
	\begin{equation}\label{weight-constant}
		[\omega]_{p,\alpha,\mathscr{D}}\leqslant c_4[\omega]_{p,\alpha}<\infty.
	\end{equation}
	
	On the other hand, we assume that $[\omega]_{p,\alpha,\mathscr{D}}<\infty$. 
	For each ball $B_\rho\subset \partial\B_n$, by Proposition \ref{dyadic-thm}, there exist an integer $1\leqslant t\leqslant N$ and a cube $Q\in\mathscr{D}_t$ such that $B_\rho\subset Q$ and $\textup{diam}_\rho(Q)\leqslant C_3\textup{diam}_\rho(B_\rho)$. 
	Then by Lemma \ref{area-ball}, there is a constant $c_5=c_5(n,p,\alpha)>0$ such that
	\begin{equation*}
		\begin{aligned}
			\frac{|\widehat{B}_\rho|_{\omega,\alpha}}{|\widehat{B}_\rho|_\alpha}\left(\frac{|\widehat{B}_\rho|_{\omega^{1-p'},\alpha}}{|\widehat{B}_\rho|_\alpha}\right)^{p-1}
			\leqslant\frac{|\widehat{Q}|_{\omega,\alpha}|\widehat{Q}|_{\omega^{1-p'},\alpha}^{p-1}}{|\widehat{B}_\rho|_\alpha^p}\leqslant c_5\frac{|\widehat{Q}|_{\omega,\alpha}|\widehat{Q}|_{\omega^{1-p'},\alpha}^{p-1}}{|\widehat{Q}|_\alpha^p}.
		\end{aligned}
	\end{equation*}
		Hence, by \eqref{new-bp},
	\[
		[\omega]_{p,\alpha}\leqslant c_5[\omega]_{p,\alpha,\mathscr{D}}<\infty.
	\]
	This completes the whole proof.
\end{proof}

\vspace{0.1cm}
\subsection{The Discretization of $P_\alpha$}
Let $N$ be the integer in Proposition \ref{dyadic-thm}.
For each integer $1\leqslant t\leqslant N$, consider the following positive integral kernel associated to the dyadic system $\mathscr{D}_t$ on $(\partial\B_n, \rho)$:
\begin{equation}\label{kernel-k}
\mathcal{K}^t_\alpha(x,y)=\sum\limits_{Q\in\mathscr{D}_t}\frac{\chi_{\widehat{Q}}(x)\chi_{\widehat{Q}}(y)}{|\widehat{Q}|_\alpha},\quad\forall x,y\in\B_n, 
\end{equation}
where $\chi_{\widehat{Q}}(x)$ is the characteristic function of $\widehat{Q}$ and $|\widehat{Q}|_\alpha=\nu_\alpha(\widehat{Q})$. 

Recall $\mathcal{R}_\alpha$ denotes the integral kernel of the $P_\alpha$, The following lemmas illustrate the relationship between $\mathcal{K}^t_\alpha(x,y)$ and $\mathcal{R}_\alpha(x,y)$.

\begin{lemma}\label{[x,y]}
	There is a constant $C_5>0$ such that for every distinct points $x,y\in\B_n$, there exists a ball $B_\rho(z,r)\subset\partial\B_n$, such that $\widehat{B}_\rho(z,r)$ contains $x,y$ and 
	$
		[x,y]\geqslant C_5 r.
	$
\end{lemma}
\begin{proof}
	Pick $x,y\in\B_n$ with $\rho(x,y)=\theta$. 
	Without loss of generality, we assume 
	$|x|\geqslant|y|$.
Using the definition of $[x,y]$, we have
	\begin{equation*}
		\begin{aligned}
			[x,y]^2&=|x-y|^2+(1-|x|^2)(1-|y|^2)\\
			&=1-2\langle x,y\rangle+|x|^2|y|^2\\
			&=\big|1-|x||y|\big|^2+2|x||y|(1-\cos\theta)\\
			&\geqslant\big|1-|y|\big|^2+4|x||y|\sin^2\frac{\theta}{2}\\
			&\geqslant\big|1-|y|\big|^2+\frac{4}{\pi^2}|y|^2 \theta^2.
		\end{aligned}
	\end{equation*}
	Let $C_5=2/\pi$ and $$r=\max\{1-|y|,\theta\},z=x/|x|\in\partial\B_n.$$
	We have $x,y\in \widehat{B}_\rho(z,r)$ and
	\[
		[x,y]\geqslant C_5 r.
	\]
	This completes the proof.
\end{proof}

\begin{lemma}\label{discretization}
	There exists a constant $C_6=C_6(n,\alpha)>0$ such that for any $x,y\in\B_n$, we have
	\[
		|\mathcal{R}_\alpha(x,y)|\leqslant C_6 \sum\limits_{t=1}^N \mathcal{K}^t_\alpha(x,y).
	\]
\end{lemma}
\begin{proof}
	For any two points $x,y\in\B_n$, by Lemma \ref{[x,y]} and Proposition \ref{dyadic-thm}, there exists a ball $B_\rho(z,r)\subset\partial\B_n$, an integer $1\leqslant t\leqslant N$ and a cube $Q\in\mathscr{D}_t$ such that $$x,y\in\widehat{B}_\rho(z,r)\subset \widehat{Q}$$ and 
	\[
	\textup{diam}_\rho(Q)\leqslant 2C_3 r \leqslant \frac{2C_3}{C_5}[x,y].
	\]
	By Lemma \ref{area-ball}, there exists a constant $c_6=c_6(n,\alpha)>0$ such that
	\begin{equation}\label{cube}
		|\widehat{Q}|_\alpha\leqslant c_6[x,y]^{n+\alpha}.
	\end{equation}
	By \cite[Theorem 1.2]{Ureyen2023}, the $\mathcal{H}$-harmonic Bergman kernel satisfies the following size estimate
\[
|\mathcal{R}_\alpha(x,y)|\leqslant \frac{c_7}{[x,y]^{n+\alpha}}, \quad\forall x,y\in\B_n,
\]
where the constant $c_7>0$ depends on $n$ and $\alpha$.
Then for any two distinct points $x,y\in\B_n$, by \eqref{cube}, we conclude the following pointwise inequality: 
	\begin{equation*}
		\begin{aligned}
			|\mathcal{R}_\alpha(x,y)|&\leqslant \frac{c_7}{[x,y]^{n+\alpha}}\\
			&\leqslant c_6c_7\frac{\chi_{\widehat{Q}}(x)\chi_{\widehat{Q}}(y)}{|\widehat{Q}|_\alpha}\\
			&\leqslant c_6c_7\sum\limits_{Q\in\mathscr{D}_t}\frac{\chi_{\widehat{Q}}(x)\chi_{\widehat{Q}}(y)}{|\widehat{Q}|_\alpha}\\
			&\leqslant c_6c_7\sum\limits_{t=1}^N \mathcal{K}^t_\alpha(x,y).
		\end{aligned}
	\end{equation*}
	Let $C_6=c_6c_7$. This completes the proof.
\end{proof}

\section{Proof of the Main Theorem}
\subsection{The case $p=2$}
The main result in this section is the following theorem:
\begin{theorem}
	Let $-1<\alpha<\infty$. 
	If $\omega$ is a $B_{2,\alpha}$ weight, then 
	\[
		P_\alpha: L^2(\omega d\nu_\alpha)\longrightarrow L^2(\omega d\nu_\alpha)
	\]
	is bounded. Moreover,  there exists a constant $C_7=C_7(n,\alpha)>0$ such that
	\[
		||P_\alpha||_{L^2(\omega d\nu_\alpha)\rightarrow L^2(\omega d\nu_\alpha)}\leqslant C_7[\omega]_{2,\alpha}.
	\]
\end{theorem}
\begin{proof}
	Assume that $\omega$ is a $ B_{2,\alpha}$ weight. Let $\{\mathscr{D}_{t}\}_{t=1}^{N}$ be the collection of dyadic systems obtianed in Proposition \ref{dyadic-thm}. By the inequality \eqref{weight-constant}, we have
\[
	[\omega]_{2,\alpha,\mathscr{D}}\leqslant c_4[\omega]_{2,\alpha}<\infty.
\]
Using Sawyer's powerful duality trick \cite{Sawyer1982}, we have
		\[
		P_\alpha: L^2(\omega d\nu_\alpha)\longrightarrow L^2(\omega d\nu_\alpha)
	\]
	 is bounded if and only if
	 \[
		P_{\omega^{-1},\alpha}: L^2(\omega^{-1} d\nu_\alpha)\longrightarrow L^2(\omega d\nu_\alpha)
	\]
	is bounded, where
	 \[
		P_{\omega^{-1},\alpha}f(x)=\int_{\B_n}f(y)\mathcal{R}_\alpha(x,y)\omega^{-1}(y)d\nu_\alpha(y)
	\]
	for $f\in L^2(\omega^{-1} d\nu_\alpha)$.
	Moreover 
	\[
		||P_\alpha||_{L^2(\omega d\nu_\alpha)\rightarrow L^2(\omega d\nu_\alpha)}=||P_{\omega^{-1},\alpha}||_{L^2(\omega^{-1} d\nu_\alpha)\rightarrow L^2(\omega d\nu_\alpha)}.
	\]
	\vspace{0.2in}
	
	For $1\leqslant t\leqslant N$, we define the integral operator on the unit ball $\B_n$ induced by the discrete kernel \eqref{kernel-k}
	\begin{equation*}
		\begin{aligned}
			T_{\omega,\alpha}^t f(x):&=\int_{\B_n}f(y)\mathcal{K}^t_\alpha(x,y)\omega(y)d\nu_\alpha(y)\\
			&=\sum\limits_{Q\in\mathscr{D}_t}\frac{1}{|\widehat{Q}|_\alpha}\int_{\widehat{Q}}f(y)\omega(y)d\nu_\alpha(y) \chi_{\widehat{Q}}(x).
		\end{aligned}
	\end{equation*}
	By Lemma \ref{discretization}, we have
	\begin{equation*}
		\begin{aligned}
			|P_{\omega^{-1},\alpha}f(x)|&\leqslant\int_{\B_n}|f(y)| |\mathcal{R}_\alpha(x,y)|\omega^{-1}(y)d\nu_\alpha(y)\\
			&\leqslant C_6\sum\limits_{t=1}^N\int_{\B_n}|f(y)|  \mathcal{K}^t_\alpha(x,y)\omega^{-1}(y)d\nu_\alpha(y)\\
			&=C_6\sum\limits_{t=1}^N (T_{\omega^{-1},\alpha}^t |f|)(x).
		\end{aligned}
	\end{equation*}
	Hence, it suffices to prove that there exists a constant $c_8>0$ such that
	\[
		||T_{\omega^{-1},\alpha}^t f||_{L^2(\omega d\nu_\alpha)}\leqslant c_8 ||f||_{L^2(\omega^{-1} d\nu_\alpha)}
	\]
	for all $0\leqslant f\in L^2(\omega^{-1} d\nu_\alpha)$ and each $1\leqslant t\leqslant N$. 
	
	Note that for a fixed dyadic system $\mathscr{D}_t$, $T_{\omega^{-1},\alpha}^t f\geqslant 0$ when $f\geqslant 0$. 
	Consequentley,
	\[
		||T_{\omega^{-1},\alpha}^t f||_{L^2(\omega d\nu_\alpha)}=\sup\limits_{0\leqslant g \in  L^2(\omega d\nu_\alpha), ||g||_{L^2(\omega d\nu_\alpha)}=1}\int_{\B_n}(T_{\omega^{-1},\alpha}^t f)(x) g(x) \omega(x) d\nu_\alpha(x).
	\]
	Now we consider the right side of the last equality. 
	Let $$0\leqslant g\in L^2(\omega d\nu_\alpha),$$ we have
	\begin{equation*}
		\begin{aligned}
			&\int_{\B_n}(T_{\omega^{-1},\alpha}^t f)(x) g(x) \omega(x) d\nu_\alpha(x)\\
			=&\sum\limits_{Q\in\mathscr{D}_t}\frac{1}{|\widehat{Q}|_\alpha}\left(\int_{\widehat{Q}}f(x) \omega^{-1}(x) d\nu_\alpha(x) \right)\left(\int_{\widehat{Q}}g(x) \omega(x) d\nu_\alpha(x) \right)\\
			=&\sum\limits_{Q\in\mathscr{D}_t}\frac{|\widehat{Q}|_\alpha}{ |\widehat{Q}|_{\omega^{-1},\alpha}|\widehat{Q}|_{\omega,\alpha}}
			\frac{|\widehat{Q}|_{\omega^{-1},\alpha}|\widehat{Q}|_{\omega,\alpha}}{|\widehat{Q}|_\alpha^2}
			\left(\int_{\widehat{Q}}f(x) \omega^{-1}(x) d\nu_\alpha(x) \right)\left(\int_{\widehat{Q}}g(x) \omega(x) d\nu_\alpha(x) \right)\\
			\leqslant& [\omega]_{2,\alpha,\mathscr{D}}\sum\limits_{Q\in\mathscr{D}_t} |\widehat{Q}|_\alpha\left(\frac{1}{ |\widehat{Q}|_{\omega^{-1},\alpha}}\int_{\widehat{Q}}f(x) \omega^{-1}(x) d\nu_\alpha(x) \right)\left(\frac{1}{ |\widehat{Q}|_{\omega,\alpha}}\int_{\widehat{Q}}g(x) \omega(x) d\nu_\alpha(x) \right).
		\end{aligned}
	\end{equation*}
	By Lemma \ref{doubling}, we have
	\[
		|\widehat{Q}|_\alpha\leqslant C_4|\widehat{Q}\backslash\widehat{Q}_{1/2}|_\alpha.
	\]
	Therefore, we have 
	\begin{equation*}
		\begin{aligned}
			&\int_{\B_n}(T_{\omega^{-1},\alpha}^t f)(x) g(x) \omega(x) d\nu_\alpha(x)\\
			\leqslant& C_4 [\omega]_{2,\alpha,\mathscr{D}}\sum\limits_{Q\in\mathscr{D}_t} |\widehat{Q}\backslash\widehat{Q}_{1/2}|_\alpha\left(\frac{1}{ |\widehat{Q}|_{\omega^{-1},\alpha}}\int_{\widehat{Q}}f(x) \omega^{-1}(x) d\nu_\alpha(x) \right)\left(\frac{1}{ |\widehat{Q}|_{\omega,\alpha}}\int_{\widehat{Q}}g(x) \omega(x) d\nu_\alpha(x) \right)\\
			=& C_4 [\omega]_{2,\alpha,\mathscr{D}}\sum\limits_{Q\in\mathscr{D}_t} \int_{\widehat{Q}\backslash\widehat{Q}_{1/2}}\left(\frac{1}{ |\widehat{Q}|_{\omega^{-1},\alpha}}\int_{\widehat{Q}}f(x) \omega^{-1}(x) d\nu_\alpha(x) \right)\left(\frac{1}{ |\widehat{Q}|_{\omega,\alpha}}\int_{\widehat{Q}}g(x) \omega(x) d\nu_\alpha(x) \right)d\nu_\alpha.
		\end{aligned}
	\end{equation*}
\vspace{0.1in}
	Note that Lemma \ref{disjoint} implies $\{\widehat{Q}\backslash\widehat{Q}_{1/2}\}_{Q\in\mathscr{D}_t}$ are disjoint. Consequently,
	\begin{equation*}
		\begin{aligned}
			&\int_{\B_n}(T_{\omega^{-1},\alpha}^t f)(x) g(x) \omega(x) d\nu_\alpha(x)\\
			\leqslant& C_4 [\omega]_{2,\alpha,\mathscr{D}}\int_{\B_n}\left(\frac{1}{ |\widehat{Q}|_{\omega^{-1},\alpha}}\int_{\widehat{Q}}f(x) \omega^{-1}(x) d\nu_\alpha(x) \right)\left(\frac{1}{ |\widehat{Q}|_{\omega,\alpha}}\int_{\widehat{Q}}g(x) \omega(x) d\nu_\alpha(x) \right)d\nu_\alpha\\
			\leqslant& C_4 [\omega]_{2,\alpha,\mathscr{D}} \int_{\B_n}(M^t_{\omega^{-1},\alpha}f)(M^t_{\omega,\alpha}g)\omega^{-\frac{1}{2}}\omega^\frac{1}{2} d\nu_\alpha\\
			\leqslant& C_4 [\omega]_{2,\alpha,\mathscr{D}} ||M_{\omega^{-1},\alpha}^t f||_{L^2(\omega^{-1} d\nu_\alpha)}||M_{\omega,\alpha}^t g||_{L^2(\omega d\nu_\alpha)}.
		\end{aligned}
	\end{equation*}
	Here $M_{\omega,\alpha}^t $ is the maximal operator associated to the dyadic system $\mathscr{D}_t$ with respect to measure $\omega d\nu_\alpha$, i.e.,
	\[
		M_{\omega,\alpha}^t f(x)=\sup\limits_{Q\in\mathscr{D}_t}\frac{\chi_{\widehat{Q}}(x)}{|\widehat{Q}|_{\omega,\alpha}}\int_{\widehat{Q}}|f|\omega d\nu_\alpha.
	\]
	
Recall that $\mathscr{D}_t$ is a dyadic system, by the classical result \cite[Theorem 5.7]{Tolsa} on the maxiaml operator on the dyadic systmes, there exists  a constant $c_9$, independent of the weight $\omega$ and the dyadic system $\mathscr{D}_t$, such that
	\[
	||M_{\omega,\alpha}^t f||_{L^2(\omega d\nu_\alpha)}\leqslant c_9||f||_{L^2(\omega d\nu_\alpha)}, \quad	\forall f\in L^2(\omega d\nu_\alpha).
	\]
	We conclude that
	\[
		\int_{\B_n}(T_{\omega^{-1},\alpha}^t f)(x) g(x) \omega(x) d\nu_\alpha(x)\leqslant C_4 c_9^2 [\omega]_{2,\alpha,\mathscr{D}} ||f||_{L^2(\omega^{-1} d\nu_\alpha)}||g||_{L^2(\omega d\nu_\alpha)}.
	\]
	Let $C_7=C_6Nc_4C_4c_9^2$, we have
	\[
		||P_\alpha||_{L^2(\omega d\nu_\alpha)\rightarrow L^2(\omega d\nu_\alpha)}\leqslant C_7[\omega]_{2,\alpha}.
	\]
	This completes the proof.
\end{proof}

\vspace{0.1cm}
\subsection{The extrapolation method for the case $p\neq2$}
The main goal of this subsection is to obtain the weighted $L^p$-estimate for $P_\alpha$. First of all, we shall adapt Rubio de Francia's extrapolation method to prove the case $p>2$ \cite{GR}. In addition, for the case $1<p<2$ follows from the following simple duality argument.
\begin{lemma}\label{lemma4.1}
	Let $1<p,p'<\infty,\frac{1}{p}+\frac{1}{p'}=1$, we have
	\begin{enumerate}[{(i)}]
		\item $||P_\alpha||_{L^p(\omega d\nu_\alpha)\rightarrow L^p(\omega d\nu_\alpha)}=||P_\alpha||_{L^{p'}(\omega^{1-p'} d\nu_\alpha)\rightarrow L^{p'}(\omega^{1-p'} d\nu_\alpha)};$
		\item $[\omega]_{p,\alpha}^{\frac{1}{p-1}}=[\omega^{1-p'}]_{p',\alpha}.$
	\end{enumerate}
\end{lemma}
The remaining proof of Theorem \ref{Thm-key} (the case $p>2$) follows from the following general known result of the  Bekoll\'e-Bonami weight $\omega$ on the real unit ball $\B_n$. For the sake of completeness and the reader’s convenience, we include a proof here.
\begin{lemma}\label{exp}
	Let $R$ be an bounded operator on $ L^2(\omega d\nu_\alpha)$. 
	If there exists a constant $C_8$ independent of the weight $\omega$ such that 
	\[
		||R||_{ L^2(\omega d\nu_\alpha)\rightarrow L^2(\omega d\nu_\alpha)}\leqslant C_8 [\omega]_{2,\alpha}\quad \mbox{for all }\omega\in  B_{2,\alpha},
	\]
	then there exists a constant $C_9$ independent of the weight $\omega$ such that
	\[
		||R||_{ L^p(\omega d\nu_\alpha)\rightarrow L^p(\omega d\nu_\alpha)}\leqslant C_9 [\omega]_{p,\alpha}\quad \mbox{for all } \omega\in  B_{p,\alpha}\mbox{ and } p>2.
	\]
\end{lemma}
\begin{proof}
For $p>2$, let
\[
\phi(p)=\frac{p-2}{p-1}
\]
then  $p'/\phi(p)=p/(p-2)$ is the dual number of $p/2$.
By the standard Rubio de Francia's extrapolation, we only need to find an operator $$D:L^{p'/\phi(p)}(\omega d\nu_\alpha) \rightarrow L^{p'/\phi(p)}(\omega d\nu_\alpha)$$ such that for all 
\[
	0\leqslant h\in L^{p'/\phi(p)}(\omega d\nu_\alpha)\quad \text{and}\quad ||h||_{L^{p'/\phi(p)}(\omega d\nu_\alpha)}=1,\]
it has following properties:
\begin{itemize}
	\item[(I)] $h\leqslant D(h)$;
	\item[(II)] $||D(h)||_{L^{p'/\phi(p)}(\omega d\nu_\alpha)}\leqslant 2$ and
	\item[(III)] $[D(h)\omega]_{2,\alpha}\leqslant C_{10}[\omega]_{p,\alpha}$, where the constant $C_{10}$ is independent of $\omega$ and $h$.
\end{itemize}
Suppose that such operator $D$ exists. Observe that for all $f\in L^p(\omega d\nu_\alpha)$, we have
\[
||Rf||^2_{L^p(\omega d\nu_\alpha)}=\sup\limits_{||h||_{L^{p'/\phi(p)}(\omega d\nu_\alpha)}=1,h\geqslant0}\int_{\B_n}|Rf|^2 h \omega d\nu_\alpha.
\]
Then
\begin{equation*}
	\begin{aligned}
		||Rf||^2_{L^p(\omega d\nu_\alpha)}
		=&\sup\limits_{||h||_{L^{p'/\phi(p)}(\omega d\nu_\alpha)}=1,h\geqslant0}\int_{\B_n}|Rf|^2 D(h) \omega d\nu_\alpha\\
		\leqslant& C_8^2\sup\limits_{||h||_{L^{p'/\phi(p)}(\omega d\nu_\alpha)}=1,h\geqslant0}\left([D(h)\omega]_{2,\alpha}^2\int_{\B_n}|f|^2 D(h) \omega d\nu_\alpha\right)\\
		\leqslant& C_8^2C_{10}^2[\omega]_{p,\alpha}^2 \sup\limits_{||h||_{L^{p'/\phi(p)}(\omega d\nu_\alpha)}=1,h\geqslant0}\left(||f||^2_{L^p(\omega d\nu_\alpha)}||D(h)||_{L^{p'/\phi(p)}(\omega d\nu_\alpha)}\right)\\
		\leqslant& 2C_8^2C_{10}^2[\omega]_{p,\alpha}^2||f||^2_{L^p(\omega d\nu_\alpha)}.
	\end{aligned}
\end{equation*}
This completes the proof up to the existence of an operator $D$.
\end{proof}

\vspace{0.1in}
Next, we show that the following maximal operator on the real unit ball $\B_n$ satisifes the properties required by the operator $D$.

Let  $M_\alpha$ be the maximal operator with respect to measure $d\nu_\alpha$, i.e.,
\[
	M_{\alpha}f(x)=\sup\limits_{B\subset\partial\B_n}\frac{\chi_{\widehat{B}}(x)}{|\widehat{B}|_\alpha}\int_{\widehat{B}}|f|d\nu_\alpha,
\]
where the supremum runs over all balls $B\subset\partial \B_n$. 
\begin{lemma}
	For $\omega\in B_{p,\alpha}$, we have
	\[
		||M_{\alpha}||_{L^p(\omega d\nu_\alpha)\rightarrow L^p(\omega d\nu_\alpha)}\leqslant C_{11}[\omega]_{p,\alpha}^{\frac{1}{p-1}},
	\]
	where $C_{11}$ is a constant dependent on $n$ and $\alpha$. 
\end{lemma}
One can see Lerner's work \cite{Lerner2008} for a simple proof.
Since $p>2$ and $[\omega]_{p,\alpha}\geqslant1$, we have
\begin{equation}\label{max}
	||M_{\alpha}||_{L^p(\omega d\nu_\alpha)\rightarrow L^p(\omega d\nu_\alpha)}\leqslant C_{11}[\omega]_{p,\alpha}^{\frac{1}{p-1}}\leqslant C_{11}[\omega]_{p,\alpha}.
\end{equation}
For $h\in L^{p'/\phi(p)}(\omega d\nu_\alpha)$, denote 
\[
	S_{\omega,\alpha}(h)=\left(\frac{M_\alpha(|h|^{1/\phi(p)}\omega)}{\omega}\right)^{\phi(p)}.
\]
\begin{lemma}\label{weight}
	For all $h\in L^{p'/\phi(p)}(\omega d\nu_\alpha)$ and weight $\omega$, we have 
	\[
		\left(\frac{1}{|\widehat{B}|_\alpha}\int_{\widehat{B}}h \omega d\nu_\alpha\right)\left(\frac{1}{|\widehat{B}|_\alpha}\int_{\widehat{B}}(S_{\omega,\alpha}(h)\omega)^{-1} d\nu_\alpha\right)\leqslant [\omega]_{p,\alpha}^{\frac{1}{p-1}}.
	\]
\end{lemma}
\begin{proof}
	Observe for $x\in\widehat{B}$
	\[
		M_\alpha (h^{1/\phi(p)}\omega)(x)\geqslant \frac{1}{|\widehat{B}|_\alpha}\int_{\widehat{B}}h^{1/\phi(p)}\omega d\nu_\alpha,
	\]
	Hence, by H\"older's inequality, we have
	\begin{equation*}
		\begin{aligned}
			&\left(\frac{1}{|\widehat{B}|_\alpha}\int_{\widehat{B}}h \omega d\nu_\alpha\right)\left(\frac{1}{|\widehat{B}|_\alpha}\int_{\widehat{B}}(S_{\omega,\alpha}(h)\omega)^{-1} d\nu_\alpha\right)\\
			=&\left(\frac{1}{|\widehat{B}|_\alpha}\int_{\widehat{B}}h \omega d\nu_\alpha\right)\left(\frac{1}{|\widehat{B}|_\alpha}\int_{\widehat{B}}(M_\alpha (h^{1/\phi(p)}\omega))^{-\phi(p)}\omega^{1-p'} d\nu_\alpha\right)\\
			\leqslant& \left( \frac{1}{|\widehat{B}|_\alpha}\int_{\widehat{B}}h^{1/\phi(p)}\omega d\nu_\alpha\right)^{\phi(p)}\left(\frac{1}{|\widehat{B}|_\alpha}\int_{\widehat{B}} \omega d\nu_\alpha\right)^{\frac{1}{p-1}}\\
			&\times \left( \frac{1}{|\widehat{B}|_\alpha}\int_{\widehat{B}}h^{1/\phi(p)}\omega d\nu_\alpha\right)^{-\phi(p)}\left(\frac{1}{|\widehat{B}|_\alpha}\int_{\widehat{B}} \omega^{1-p'} d\nu_\alpha\right)\\
			\leqslant&[\omega]_{p,\alpha}^{\frac{1}{p-1}}.
		\end{aligned}
	\end{equation*}
\end{proof}

Furthermore, by \eqref{max}, we have
\begin{equation*}
	\begin{aligned}
		||S_{\omega,\alpha}(h)||_{L^{p'/\phi(p)}(\omega d\nu_\alpha)}
		&=\left(\int_{\B_n}|M_\alpha(|h|^{1/\phi(p)}\omega)|^{p'} \omega^{1-p'} d\nu_\alpha\right)^{\phi(p)/p'}\\
		&\leqslant \left(||M_{\alpha}||_{L^{p'}(\omega^{1-p'} d\nu_\alpha)\rightarrow L^{p'}(\omega^{1-p'} d\nu_\alpha)}\left\||h|^{1/\phi(p)}\omega\right\|_{L^{p'}(\omega^{1-p'} d\nu_\alpha)}\right)^{\phi(p)}\\
		&\leqslant ||M_{\alpha}||_{L^p(\omega d\nu_\alpha)\rightarrow L^p(\omega d\nu_\alpha)}^{\phi(p)}\|h\|_{L^{p'/\phi(p)}(\omega d\nu_\alpha)}\\
		&\leqslant C_{11}^{\phi(p)}[\omega]_{p,\alpha}^{\phi(p)}\|h\|_{L^{p'/\phi(p)}(\omega d\nu_\alpha)}.
	\end{aligned}
\end{equation*}
Denote 
\begin{equation}\label{norm}
	A=\sup\limits_{||h||_{L^{p'/\phi(p)}(\omega d\nu_\alpha)}=1}||S_{\omega,\alpha}(h)||_{L^{p'/\phi(p)}(\omega d\nu_\alpha)}\leqslant C_{11}^{\phi(p)}[\omega]_{p,\alpha}^{\phi(p)}.
\end{equation}
We define an operator 
\[
	D(h)=\sum\limits_{k=0}^\infty\frac{1}{2^k}\frac{(S_{\omega,\alpha})^k(h)}{A^k},
\]
where $(S_{\omega,\alpha})^k$ is k-times composition of $S_{\omega,\alpha}$ and $(S_{\omega,\alpha})^0$ is identity operator.
It is easy to see $|h|\leqslant D(h)$. 
Now we prove another what we want.
\begin{lemma}\label{2 to p}
	For all $h\geqslant0, ||h||_{L^{p'/\phi(p)}(\omega d\nu_\alpha)}=1$, we have
	\[
		[D(h)\omega]_{2,\alpha}\leqslant 2C_{11}^{\phi(p)}[\omega]_{p,\alpha}.
	\]
\end{lemma}
\begin{proof}
	We have two observations:
	\begin{enumerate}
		\item $||D(h)||_{L^{p'/\phi(p)}(\omega d\nu_\alpha)}\leqslant 2||h||_{L^{p'/\phi(p)}(\omega d\nu_\alpha)}$;
		\item $S_{\omega,\alpha}(D(h))\leqslant 2AD(h)$.
	\end{enumerate}
	Hence, by Lemma \ref{weight} and \eqref{norm},
	\begin{equation*}
		\begin{aligned}
			&\left(\frac{1}{|\widehat{B}|_\alpha}\int_{\widehat{B}}D(h)\omega d\nu_\alpha\right)\left(\frac{1}{|\widehat{B}|_\alpha}\int_{\widehat{B}}(D(h)\omega)^{-1} d\nu_\alpha\right)\\
			\leqslant&2A\left(\frac{1}{|\widehat{B}|_\alpha}\int_{\widehat{B}}D(h)\omega d\nu_\alpha\right)\left(\frac{1}{|\widehat{B}|_\alpha}\int_{\widehat{B}}(S_{\omega,\alpha}(D(h))\omega)^{-1} d\nu_\alpha\right)\\
			\leqslant&2A[\omega]_{p,\alpha}^{\frac{1}{p-1}}\\
			\leqslant& 2C_{11}^{\phi(p)}[\omega]_{p,\alpha}.
		\end{aligned}
	\end{equation*}
\end{proof}

We conclude that for a weight $\omega\in B_{p,\alpha}$ and $0\leqslant h\in L^{p'/\phi(p)}(\omega d\nu_\alpha),||h||_{L^{p'/\phi(p)}(\omega d\nu_\alpha)}=1$, we have $D(h)\omega\in B_{2,\alpha}.$ 
Then by Lemma \ref{2 to p}, we have
\begin{equation*}
	\begin{aligned}
		 \int_{\B_n}|Rf|^2 h \omega d\nu_\alpha\leqslant&\int_{\B_n}|Rf|^2 D(h) \omega d\nu_\alpha\\
		\leqslant&C_8^2[D(h)\omega]_{2,\alpha}^2 \int_{\B_n}|f|^2 D(h) \omega d\nu_\alpha\\
		\leqslant&2C_8^2C_{11}^{2\phi(p)}[\omega]_{p,\alpha}^2||f||^2_{L^p(\omega d\nu_\alpha)}||D(h)||_{L^{p'/\phi(p)}(\omega d\nu_\alpha)}\\
		\leqslant&4C_8^2C_{11}^{2\phi(p)}[\omega]_{p,\alpha}^2||f||^2_{L^p(\omega d\nu_\alpha)}.
	\end{aligned}
\end{equation*}
This completes the proof.

\section{The sharp linear estimate for $P_\alpha$}\label{S:sharp}
By Theorem \ref{Thm-key}, if $\omega$ is a $B_{2,\alpha}$ weight, then there is a constant $C$ such that
\[
 ||P_\alpha||_{ L^2(\omega d\nu_\alpha)\to L^2(\omega d\nu_\alpha)}\leqslant C  [\omega]_{2,\alpha}.
\]
It follows that the norm $ ||P_\alpha||_{ L^2(\omega d\nu_\alpha)\to L^2(\omega d\nu_\alpha)}$ has linear growth with respect to the Bekoll\'e-Bonami constant. In this section, we shall show that this linear growth is sharp.

\vspace{0.1in}
For $n\geqslant 3,$ let $-1<\alpha<\infty$ and $0<\delta<1$. 
Consider the weight on the unit ball $\B_n$ defined in Subsection \ref{S:sharp-example}: $$\omega(x)=(1-|x|^2)^s$$ with $s=(\alpha+1)(1-\delta)$. 

By the \eqref{example-weight}, there is a constant $c_3>0$ such that the Bekoll\'e-Bonami constant has an upper estimate:
\[
	[\omega]_{2,\alpha}\leqslant\frac{c_3}{\delta}.
\]

Let $r_0=1$, $e_1=(1,0,\dots,0)$ and $B_0=B_\rho(e_1,r_0).$
We consider a function $$f(x)=(1-|x|^2)^{-s}\chi_{\widehat{B_0}}(x).$$
Then Lemma \ref{area-ball} yields the norm estimate
\begin{equation*}
	\begin{aligned}
		||f||_{L^2(\omega d\nu_\alpha)}^2&= \frac{c_\alpha}{c_{\alpha-s}}|\widehat{B}_0|_{\alpha-s}\\
		&\leqslant \frac{c_{10}}{\delta},
	\end{aligned}
\end{equation*}
where $c_{10}=C_2c_\alpha/(\alpha+1).$

\vspace{0.1in}
For $\gamma>0$ and $\xi\in\partial\B_n$, we define the nontangential approach region with vertex $\xi$:
\[
	\Omega_\gamma(\xi):=\{y\in\B_n:0<\langle y,\xi \rangle<1,|y|^2-|\langle y,\xi \rangle|^2<\gamma^2(1-|\langle y,\xi \rangle|)^2\}.
\]
In fact, $\Omega_\gamma(\xi)$ is a truncated circular cone in $\B_n$ with vertex at $\xi$.
Note 
\[\sin(\arctan(\gamma))=\frac{\gamma}{\sqrt{1+\gamma^2}}.\]
Then there exists half a ball with center at $0$ and radius $\gamma/\sqrt{1+\gamma^2}$ contained in $\Omega_\gamma(\xi)$.
Let
\[
	G=\left\{x\in\B_n:\rho\left(\frac{x}{|x|},e_1\right)<\frac{\pi}{2}-r_0, |x|<\frac{\gamma}{\sqrt{1+\gamma^2}}\right\}.
\]
Then 
\[
G\subset \bigcap\limits_{\xi\in B_0}\Omega_\gamma(\xi).
\]
For some fixed $\gamma<1/2$, by \"Ureyen \cite[Proposition 6.2]{Ureyen2023}, there exists a constant $c_{11}=c_{11}(n,\alpha,\gamma)>0$ such that  for all $x\in G$ and $y\in \widehat{B}_0$:
\[
	\mathcal{R}_\alpha(x,y)\geqslant c_{11}.
\]
Then by Lemma \ref{area-ball}, it yields the lower bound when $x\in G$
\begin{equation*}
	\begin{aligned}
		P_\alpha(f)(x)&=\int_{\B_n}\mathcal{R}_\alpha(x,y) f(y) d\nu_\alpha(y)\\
		&\geqslant c_{11}\frac{c_\alpha}{c_{\alpha-s}}|\widehat{B}_0|_{\alpha-s}\\
		&\geqslant \frac{c_{12}}{\delta}.
	\end{aligned}
\end{equation*}
Here $c_{12}=C_1c_\alpha c_{11}/(\alpha+1)$.
\vspace{0.1in}

Therefore, we have
\begin{equation*}
	\begin{aligned}
		||P_\alpha(f)||_{L^2(\omega d\nu_\alpha)}^2&\geqslant\int_{G}|P_\alpha(f)(x)|^2 \omega(x) d\nu_\alpha(x)\\
		&\geqslant\frac{c_{12}|G|_\alpha}{\delta^2}\\
		&\geqslant\frac{c_{12}|G|_\alpha}{c_3c_{10}}[\omega]_{2,\alpha} ||f||_{L^2(\omega dA_\alpha)}^2.
	\end{aligned}
\end{equation*}
This completes the proof.

\vspace{0.5in}
\noindent\textbf{Acknowledgement} This work was supported by the National Natural Science Foundation of China (Grant No. 12231005;12471116), and National Key R$\&$D Program of China (2024YFA1013400).

\end{document}